\documentclass[11pt,a4paper]{amsart}

\usepackage[utf8]{inputenc}
\usepackage[T1]{fontenc}
\usepackage{lmodern}

\usepackage{xcolor}
\usepackage[normalem]{ulem}

\usepackage{tikz}
\usepackage{tikz-cd}

\usepackage[english]{babel}
\usepackage{microtype}

\usepackage{amssymb, amsthm}
\usepackage{mathtools}

\usepackage[colorlinks=true, linkcolor=blue, citecolor=blue, urlcolor=blue]{hyperref}
\usepackage{cleveref}

\usepackage{geometry}
\numberwithin{equation}{section}

\newtheorem{theorem}{Theorem}[section]
\newtheorem{lemma}[theorem]{Lemma}

\theoremstyle{definition}

\theoremstyle{remark}
\newtheorem{remark}[theorem]{Remark}

\title[Solvability of the B\'ezout Equation for  \( H^\infty \) Functions on the Polydisk]{Solvability of the B\'ezout Equation for Banach Algebra--Valued \( H^\infty \) Functions on the Polydisk}

\author[A.~Brudnyi]{Alexander Brudnyi}
\address{Department of Mathematics and Statistics, University of Calgary, Calgary, Alberta, Canada}
\email{abrudnyi@ucalgary.ca}

\author[M.~Withanachchi]{Mahishanka Withanachchi}
\address{Department of Mathematics and Statistics, University of Calgary, Calgary, Alberta, Canada}
\email{mahishanka.withanach@ucalgary.ca}


\date{}

\begin{document}

\begin{abstract}
In connection with the still unsolved multidimensional corona problem for algebras of bounded holomorphic functions on convex domains, we study the solvability of the B\'ezout equation for the algebra of bounded holomorphic functions on the polydisk with values in a complex Banach algebra. Assuming local solvability of the B\'ezout equation on a special open cover of the maximal ideal space of the algebra, we combine a dimension-induction scheme with a careful analysis of the topological structure of this space to glue local solutions into a global one. As a corollary, we obtain the solvability of the B\'ezout equation for a broader class of subalgebras containing the slice algebra of bounded holomorphic functions, the case of the latter having been previously proved by the first author.
\end{abstract}

\maketitle

\section{Introduction}

Let \(A\) be a unital commutative complex Banach algebra. Its maximal ideal space 
\(M(A)\subset A^*\) is the set of nonzero homomorphisms \(A\to\mathbb C\). Every element of \(M(A)\) 
is a continuous linear functional on \(A\) of norm at most one, so \(M(A)\) lies in the closed unit ball of \(A^*\). 
It is a compact Hausdorff space in the weak-\(\ast\) topology of \(A^*\). 
The Gelfand transform
\[
\hat a(\xi):=\xi(a), \qquad a\in A,\ \xi\in M(A),
\]
is a norm nonincreasing morphism of Banach algebras \(A\to C(M(A))\). 
If \(A\subset C_b(X)\), the algebra of bounded continuous functions on a Hausdorff space \(X\) with the supremum norm, 
and \(A\) separates points of \(X\), then the Gelfand transform is an isometric embedding. 
In this case, the evaluation map
\[
\iota:X\hookrightarrow M(A), \qquad \iota(x)(f)=f(x),\ f\in A,
\]
is continuous, and the complement of the closure of \(\iota(X)\) in \(M(A)\) is called the \emph{corona} of \(A\).  

The \emph{corona problem} asks whether this complement is empty.
Equivalently (cf.~Garnett \cite[Ch.~V, Thm.~1.8]{Garnett1981}), it asks whether for every finite collection
\(f_1,\dots,f_k\in A\) satisfying the corona condition
\begin{equation}\label{e1}
\inf_{x\in X}\ \max_{1\le j\le k} |f_j(x)|>0,
\end{equation}
the corresponding Bézout equation is solvable, that is, whether there exist
\(g_1,\dots,g_k\in A\) such that
\begin{equation}\label{e2}
f_1 g_1+\cdots+f_k g_k=1 \qquad \text{on } X.
\end{equation}

\smallskip

For \(A=H^\infty(\mathbb D)\), the algebra of bounded holomorphic functions on the unit disk,  
the problem was raised by Kakutani in 1941. Newman \cite{Newman1959} introduced the term \emph{corona},  
and showed that the problem is equivalent to an interpolation question later solved by Carleson \cite{Carleson1962}.  
Carleson proved that the corona is empty, equivalently that every collection of bounded holomorphic functions on \(\mathbb D\) satisfying \eqref{e1} admits a solution to \eqref{e2}.  
Hörmander \cite{Hormander1967} introduced \(\bar\partial\)-methods together with the Koszul complex to simplify the proof,  
and Wolff \cite{Wolff1980} later gave the shortest argument by directly obtaining uniform estimates on the boundary circle of the solution of the relevant \(\bar\partial\)-equation (see also Garnett \cite{Garnett1981}).

Following Carleson’s theorem, the corona problem was settled for many classes of Riemann surfaces: 
finite bordered surfaces \cite{Alling1964,Alling1965,EarleMarden1968,Forelli1966,St1,St2,St3,HaraNakay1985,Oh2008}, 
certain infinitely connected domains (``roadrunner'' domains) by Behrens \cite{Behrens1970,Behrens1971} (extended in \cite{Deeb1977a,DeebWilken1977,Narita1985}), 
and finitely sheeted coverings by Nakai \cite{Nakay1983}. 
It was later shown, starting with the work of Cole (see Gamelin \cite{Gamelin1978}), 
that there exist Riemann surfaces with nontrivial corona 
(see also Barrett--Diller \cite{BarrettDiller1998}, Lárusson \cite{Larusson2000}).

For planar domains the general corona problem is still unsolved. In this direction, Gamelin \cite{Gamelin1970} has shown that the corona problem for planar domains is local in the sense that it depends only on the behavior of the domain locally about each boundary point.
Moore \cite{Moore1987} proved that the corona is empty for domains whose boundary lies on the graph of a \(C^{1+\varepsilon}\) function, 
extending the earlier result of Garnett--Jones \cite{GarnettJones1985} for Denjoy domains. 
Further extensions include Handy’s theorem for complements of certain Cantor sets \cite{Handy2009} and NewDelman’s result for subsets of Lipschitz graphs \cite{NewDelman2011}.  

The problem also remains open in several variables for the ball and the polydisk. 
In fact, no examples are known of domains in \(\mathbb C^n\), \(n\ge2\), whose algebra \(H^\infty\) has a nontrivial corona.

\section{Main Results}\label{main}

In this paper we study the corona problem for the Banach algebra \( H^\infty(\mathbb{D}^n) \) of bounded holomorphic functions on the polydisk \( \mathbb{D}^n \subset \mathbb{C}^n \). As before, \( M(H^\infty(\mathbb{D}^n)) \) denotes its maximal ideal space, equipped with the weak-* topology.

Given a unital commutative complex Banach algebra \( (A, \lVert\cdot\rVert_A) \), we write \( H^\infty(\mathbb{D}^n, A) \) for the Banach algebra of bounded holomorphic functions from \( \mathbb{D}^n \) into \( A \), equipped with the norm
\begin{equation}\label{e2.1}
\|f\|_\infty = \sup_{z \in \mathbb{D}^n} \|f(z)\|_A.
\end{equation}

One of our main results is the following:

\begin{theorem}\label{te2.1}
Let \( A \) be a unital commutative complex Banach algebra, and let \( f_1, \dots, f_k \in H^\infty(\mathbb{D}^n, A) \). Suppose there exist open sets \( \widehat{U}_1, \dots, \widehat{U}_k \subset [M(H^\infty(\mathbb{D}))]^n \) such that
\[
[M(H^\infty(\mathbb{D}))]^n = \bigcup_{i=1}^k \widehat{U}_i, \quad \text{and} \quad U_i := \mathbb{D}^n \cap \widehat{U}_i.
\]
Assume that for each \( i \), there exist \( g_{i1}, \dots, g_{ik} \in H^\infty(U_i, A) \) such that
\[
\sum_{j=1}^k f_j g_{ij} = 1 \quad \text{on } U_i.
\]
Then there exist \( g_1, \dots, g_k \in H^\infty(\mathbb{D}^n, A) \) such that
\[
\sum_{j=1}^k f_j g_j = 1 \quad \text{on } \mathbb{D}^n.
\]
\end{theorem}

\smallskip

Next, we study the structure of the maximal ideal space \( M(H^\infty(\mathbb{D}^n)) \). To this end, we consider the natural continuous projection
\begin{equation}\label{e2.2}
\pi_n : M(H^\infty(\mathbb{D}^n)) \longrightarrow \bigl[ M(H^\infty(\mathbb{D})) \bigr]^n,
\end{equation}
defined as the transpose of the canonical isometric embedding
\[
\widehat{\otimes}_\varepsilon^n H^\infty(\mathbb{D}) \hookrightarrow H^\infty(\mathbb{D}^n),
\]
where \( \widehat{\otimes}_\varepsilon^n H^\infty(\mathbb{D}) \) denotes the \( n \)-fold injective (or \(\varepsilon\)-) tensor product of \( H^\infty(\mathbb{D}) \), regarded as the closed subalgebra of \( H^\infty(\mathbb{D}^n) \) generated by bounded holomorphic functions in the individual coordinate variables. In particular, \( \pi_n \) restricts to the identity on \( \mathbb{D}^n \subset M(H^\infty(\mathbb{D}^n)) \).

For each point \( \xi \in \bigl[ M(H^\infty(\mathbb{D})) \bigr]^n \), define the corresponding fiber
\begin{equation}\label{e2.3}
F_\xi := \pi_n^{-1}(\xi) \cap \operatorname{cl}(\mathbb{D}^n) \subset M(H^\infty(\mathbb{D}^n)),
\end{equation}
where \( \operatorname{cl}(\mathbb{D}^n) \) denotes the closure of \( \mathbb{D}^n \) in \( M(H^\infty(\mathbb{D}^n)) \).

We equip \( \operatorname{cl}(\mathbb{D}^n) \) with the topology induced from \( M(H^\infty(\mathbb{D}^n)) \). Given an open subset \( U \subset \operatorname{cl}(\mathbb{D}^n) \), we write \( \mathcal{O}(U) \subset C(U) \) for the algebra of continuous functions on \( U \) that are holomorphic on \( U \cap \mathbb{D}^n \) in the usual sense.

For each fiber \( F_\xi \), we define the associated uniform algebra \( A(F_\xi) \subset C(F_\xi) \) as the uniform closure of the algebras \( \mathcal{O}(U)\rvert_{F_\xi} \), where \( U \) ranges over all open neighborhoods of \( F_\xi \) in \( \operatorname{cl}(\mathbb{D}^n) \).

Since \( F_\xi \) is compact, the family of sets of the form
\[
\pi_n^{-1}(V) \cap \operatorname{cl}(\mathbb{D}^n),
\]
with \( V \) ranging over open neighborhoods of \( \xi \in \bigl[ M(H^\infty(\mathbb{D})) \bigr]^n \), forms a basis of neighborhoods of \( F_\xi \) in \( \operatorname{cl}(\mathbb{D}^n) \); that is, every open neighborhood of \( F_\xi \) contains a set of this form.
Moreover, each such set \( V \) contains a compact subset of the form \( W = W_1 \times \dots \times W_n \), where, for \( \xi = (\xi_1, \dots, \xi_n) \in \bigl[ M(H^\infty(\mathbb{D})) \bigr]^n \), each component \( W_i \subset M(H^\infty(\mathbb{D})) \) is given by
\[
W_i := \left\{ \eta \in M(H^\infty(\mathbb{D})) : |\hat{f}(\eta)| \le 1 \right\},
\]
for some function \( f \in H^\infty(\mathbb{D}) \) satisfying \( \hat{f}(\xi_i) = 0 \). Such a choice is guaranteed by the classical result of Suárez \cite[Thm.~2.4]{Suarez1994}, which asserts that the algebra \( H^\infty(\mathbb{D}) \) is {\em separating}.

In particular, the set
\[
W \cap \mathbb{D}^n = (W_1 \cap \mathbb{D}) \times \dots \times (W_n \cap \mathbb{D})
\]
satisfies the hypotheses of the Runge-type approximation theorem established in~\cite{Brudnyi2026}. That result asserts that every function in \( H^\infty\bigl( \mathbb{D}^n \cap \pi_n^{-1}(V) \bigr) \) can be uniformly approximated on \( W \cap \mathbb{D}^n \) by functions from \( H^\infty(\mathbb{D}^n) \).

Summarizing the above, we obtain that for every open neighborhood \( U \subset \operatorname{cl}(\mathbb{D}^n) \) of \( F_\xi \), the algebra \( \mathcal{O}(U)\rvert_{F_\xi} \) is contained in the uniform closure of the algebra \( \mathcal{O}\bigl( M(H^\infty(\mathbb{D}^n)) \bigr)\rvert_{F_\xi} \), where
\[
\mathcal{O}\bigl( M(H^\infty(\mathbb{D}^n)) \bigr) := \widehat{H^\infty(\mathbb{D}^n)} = \left\{ \hat{f} : f \in H^\infty(\mathbb{D}^n) \right\}.
\]
In view of the definition of the algebra \( A(F_\xi) \), it follows that \( A(F_\xi) \) coincides with the uniform closure of the algebra  \( \widehat{H^\infty(\mathbb{D}^n)}\rvert_{F_\xi} \). It then follows from a standard result in the theory of uniform algebras that the maximal ideal space of \( A(F_\xi) \) coincides with the \( H^\infty \)-hull of \( F_\xi \), namely
\begin{equation}\label{e2.4}
M(A(F_\xi)) 
   = \left\{ \zeta \in M(H^\infty(\mathbb{D}^n)) : 
      |\hat{f}(\zeta)| \le \sup_{F_\xi} |\hat{f}| 
      \ \text{for all } f \in H^\infty(\mathbb{D}^n) \right\}.
\end{equation}

\begin{theorem}\label{teo2.2}
For each \( \xi \in \left[ M(H^\infty(\mathbb{D})) \right]^n \), we have
\[
\pi_n^{-1}(\xi) = M(A(F_\xi)).
\]

The corona problem for \( H^\infty(\mathbb{D}^n) \) is solvable {\rm (}i.e.,
\( \operatorname{cl}(\mathbb{D}^n) = M(H^\infty(\mathbb{D}^n)) \){\rm )}
if and only if
\[
M(A(F_\xi)) = F_\xi \quad \text{for all } \xi \in \left[ M(H^\infty(\mathbb{D})) \right]^n.
\]
\end{theorem}

To describe certain fibers \( F_\xi \) more explicitly, we begin by recalling classical results of K.~Hoffman on the structure of the maximal ideal space \( M(H^\infty(\mathbb{D})) \).

The pseudohyperbolic metric on the unit disk \( \mathbb{D} \) is given by
\begin{equation}\label{e2.5}
\rho(z, w) := \left| \frac{z - w}{1 - \bar{w} z} \right|, \quad z, w \in \mathbb{D},
\end{equation}
and admits a natural extension to \( M(H^\infty(\mathbb{D})) \) via the formula
\begin{equation}\label{e2.6}
\rho(\xi, \eta) := \sup \left\{ |\hat{f}(\eta)| : f \in H^\infty(\mathbb{D}),\ \hat{f}(\xi) = 0,\ \lVert f \rVert_\infty \le 1 \right\}, \quad \xi, \eta \in M(H^\infty(\mathbb{D})).
\end{equation}

The \emph{Gleason part} of a point \( x \in M(H^\infty(\mathbb{D})) \) is 
\begin{equation}\label{e2.7}
 \pi(\xi) := \left\{ \eta \in M(H^\infty(\mathbb{D})) : \rho(\xi, \eta) < 1 \right\}. 
 \end{equation} 
 These parts form a partition of \( M(H^\infty(\mathbb{D})) \): for any \( \xi, \eta \), either \( \pi(\xi) = \pi(\eta) \) or \( \pi(\xi) \cap \pi(\eta) = \emptyset \).
\smallskip

A fundamental result of Hoffman \cite{Hoffman1967} asserts that each Gleason part in \( M(H^\infty(\mathbb{D})) \) is either a singleton or an analytic disk (i.e., the image of a continuous injective map \( \tau\colon \mathbb{D} \to M(H^\infty(\mathbb{D})) \) such that \( \hat{f} \circ \tau \in H^\infty(\mathbb{D}) \) for all \( f \in H^\infty(\mathbb{D}) \)). A point \( \xi \in M(H^\infty(\mathbb{D})) \) lies in a nontrivial part if and only if it belongs to the closure of an interpolating sequence. The union \( M_a \) of all such parts is open in the maximal ideal space, and its complement \( M_s := M(H^\infty(\mathbb{D})) \setminus M_a \) is closed and totally disconnected by a theorem of Suárez~\cite{Suarez1996}.

\smallskip

The space \( [M(H^\infty(\mathbb{D}))]^n \) admits a natural stratification according to the number of coordinates lying in the open subset \( M_a \) of nontrivial Gleason parts. Each stratum is determined by the subset of indices \( j \in \{1,\dots,n\} \) for which the \( j \)-th coordinate belongs to \( M_a \) (and hence the remaining coordinates lie in \( M_s \)).

\smallskip

We now describe the structure of fibers of the map
\[
\pi_n: M(H^\infty(\mathbb{D}^n)) \to [M(H^\infty(\mathbb{D}))]^n
\]
over the first two strata: (i) the open stratum \( M_a^n \), and (ii) the codimension-one stratum consisting of those \( \xi \in [M(H^\infty(\mathbb{D}))]^n \) with exactly one coordinate in \( M_s \) and the remaining \( n-1 \) in \( M_a \).

\begin{theorem}\label{teo2.3}
Let \( \xi \in \bigl[M\bigl(H^\infty(\mathbb{D})\bigr)\bigr]^n \). Then:

\begin{itemize}
\item[(i)] If \( \xi \in M_a^n \), then the fiber \( \pi_n^{-1}(\xi) \) coincides with 
\[
F_\xi := \pi_n^{-1}(\xi) \cap \operatorname{cl}(\mathbb{D}^n),
\]
and \( F_\xi \) lies in the closure of an interpolating sequence for \( H^\infty(\mathbb{D}^n) \).\footnote{A sequence \( \{z_k\}\subset\mathbb{D}^n \) is interpolating for \(H^\infty(\mathbb{D}^n)\) if for every bounded sequence \( \{a_k\}\in\ell^\infty \) there exists \( f\in H^\infty(\mathbb{D}^n) \) such that \( f(z_k)=a_k \) for all \( k \).} In particular, \(F_\xi\) is totally disconnected.

\item[(ii)] If exactly one coordinate of \( \xi \) lies in \( M_s \) and the remaining \( n-1 \) in \( M_a \), then the fiber \( \pi_n^{-1}(\xi) \) has covering dimension \( \le 2 \) and decomposes as the disjoint union of a compact totally disconnected subset and a (possibly empty) open subset consisting of pairwise disjoint analytic disks.
\end{itemize}

\noindent In both cases, \( M\bigl(A(F_\xi)\bigr) = F_\xi \).
\end{theorem}

\smallskip

In case \textup{(ii)}, the fiber \( F_\xi \) is a disjoint union of Gleason parts for the algebra
\( H^\infty(\mathbb{D}\times\mathbb{N}) \) of bounded holomorphic functions on the disjoint countable
union of open disks. Each such part is either an analytic disk or a single point.
For a detailed study of the structure of \( M\bigl(H^\infty(\mathbb{D}\times\mathbb{N})\bigr) \), 
see, e.g., \cite{Brudnyi2021}.
\smallskip

\noindent Recall that the \emph{covering dimension} of a normal space \(X\), denoted 
\(\dim X\), is defined as follows:
\begin{itemize}
  \item \(\dim X < n\) if every open cover of \(X\) has a finite open refinement 
  of order \( \le n \) (i.e., no point lies in more than \(n+1\) sets of the refinement);
  \item \(\dim X = n\) if \(\dim X < n+1\) and \(\dim X \nless n\);
  \item \(\dim X = \infty\) if no such integer \(n\) exists.
\end{itemize}

\noindent We mention that every compact Hausdorff totally disconnected space has covering dimension \(0\).

\smallskip
\noindent
See \cite[Ch.~1, §6.7]{Engelking1989}.

\begin{remark}\label{rem2.4}
For strata of \( [M(H^\infty(\mathbb{D}))]^n \) of codimension between \(2\) and \(n-1\), which arise only for \(n \ge 3\), the structure of fibers \( F_\xi \) is more subtle. Nevertheless, one can still show that \( M(A(F_\xi)) = F_\xi \) provided the corona theorem holds for the algebras \( H^\infty(\mathbb{D}^k) \) with \( k = 2, 3, \dots, n-1 \).

The main difficulty arises for the lowest stratum \( M_s^n \). In this case, the fiber \( F_\xi \) may contain the image of an injective continuous map \( \tau\colon \mathbb{D}^n \to \operatorname{cl}(\mathbb{D}^n) \) such that \( \hat f \circ \tau \in H^\infty(\mathbb{D}^n) \) for all \( f \in H^\infty(\mathbb{D}^n) \). The existence of such ‘embedded’ \( n \)-dimensional analytic structures obstructs an inductive approach to proving the corona theorem, as no reduction in dimension is possible. We will explore this phenomenon in more detail in a forthcoming paper.
\end{remark}

We now illustrate the scope of Theorems~\ref{te2.1} and~\ref{teo2.2} by extending
the class of subalgebras of $H^\infty(\mathbb D^n)$ over which the
corresponding B\'ezout equations admit global solutions, and by showing that
such solvability is preserved under all perturbations in $H^\infty(\mathbb D^n)$
up to the maximal size allowed by the corona condition.

Let \(S_n(H^\infty(\mathbb D))\subset H^\infty(\mathbb D^n)\) denote the slice
algebra, i.e. the closed subalgebra consisting of those holomorphic functions
whose boundary values on
\[
\mathbb T^n := \{ z=(z_1,\dots,z_n)\in\mathbb C^n : |z_i|=1,\ \text{for all } i \}
\]
belong to 
\[
\widehat\otimes_{\varepsilon}^{\,n} L^\infty(\mathbb T)\subset L^\infty(\mathbb T^n).
\]
Equivalently, these are the holomorphic functions on \(\mathbb D^n\) that can be
uniformly approximated by bounded functions which are \emph{separately harmonic}
in each variable.

It is not known in general whether \(S_n(H^\infty(\mathbb D))\) coincides with
the uniform algebra generated by \(H^\infty(\mathbb D)\) in each coordinate, that is,
whether
\[
S_n\bigl(H^\infty(\mathbb D)\bigr)
   = \widehat\otimes_\varepsilon^n H^\infty(\mathbb D),
\]
a question related to the longstanding open problem of whether \(H^\infty(\mathbb D)\)
has the approximation property (see Lindenstrauss~\cite{Lindenstrauss1976},
Question~7; see also Bourgain--Reinov~\cite{BourgainReinov1985} for related results).
\medskip

Let \(H_1,\dots,H_m\in H^\infty(\mathbb D^n)\setminus S_n\bigl(H^\infty(\mathbb D)\bigr)\), and set  
\[
H:=(H_1,\dots,H_m):\mathbb D^n\longrightarrow\mathbb C^m.
\]
We denote by
\[
A(S_n;H)
   := \langle\, S_n\bigl(H^\infty(\mathbb D)\bigr),\; H_1,\dots,H_m\,\rangle
   \subset H^\infty(\mathbb D^n)
\]
the uniform closure of the algebra generated by the slice algebra and the 
functions \(H_1,\dots,H_m\).

\begin{theorem}\label{te2.5}
Assume that for every point \(x\in\mathbb D^n\) there exists a neighborhood 
\(U_x\subset\mathbb D^n\) of \(x\) and a one--dimensional complex analytic subset 
\(\Sigma_x\subset\mathbb C^m\) such that \(H(U_x)\subset\Sigma_x\).

Fix \(\delta>0\) and suppose \(g_1,\dots,g_k\in A(S_n;H)\) satisfy the corona condition
\[
\inf_{z\in\mathbb D^n}\ \max_{1\le i\le k}|g_i(z)| \ \ge\ \delta > 0.
\]
Then for every \(h_1,\dots,h_k\in H^\infty(\mathbb D^n)\) such that
\[
\max_{1\le i\le k}\,\|h_i\|_\infty < \delta,
\]
the B\'ezout equation
\[
(g_1+h_1)f_1+\cdots+(g_k+h_k)f_k = 1 
\quad \text{on }\ \mathbb D^n
\]
admits a solution \(f_1,\dots,f_k\in H^\infty(\mathbb D^n)\).
\end{theorem}

\begin{remark}\label{rem2.6}\rm
(a) 
The hypothesis of Theorem~\ref{te2.5} on the map
\(H=(H_1,\dots,H_m)\) is equivalent to the existence of holomorphic
maps
\[
h:\mathbb D^n\to\mathbb D,
\qquad
\widetilde H:\mathbb D\to\mathbb C^m
\]
with \(H=\widetilde H\circ h\); see Lemma~\ref{lem6.1}. In particular, Theorem~\ref{te2.5} applies whenever the components
of \(H\) factor through a holomorphic map from \(\mathbb D^n\) to \(\mathbb D\).

\smallskip

(b) 
If \(g_1,\dots,g_k\in S_n\bigl(H^\infty(\mathbb D)\bigr)\) satisfy the corona 
condition
\[
\inf_{z\in\mathbb D^n}\ \max_{1\le i\le k}|g_i(z)|>0,
\]
then the B\'ezout equation
\[
g_1f_1+\cdots+g_kf_k = 1
\quad \text{on }\mathbb D^n
\]
admits a solution \(f_1,\dots,f_k\in S_n\bigl(H^\infty(\mathbb D)\bigr)\).
This was established in \cite[Cor.~1.16]{Brudnyi2013} as a consequence of the 
more general result \cite[Th.~1.12]{Brudnyi2013}.
We note that Theorem~\ref{te2.5} is significantly stronger: it applies already 
when \(H\) consists of a single function from 
\(H^\infty(\mathbb D^n)\setminus S_n\bigl(H^\infty(\mathbb D)\bigr)\).

\smallskip

(c)
The perturbation condition in Theorem~\ref{te2.5} is optimal.  
Indeed, assume
\[
\inf_{z\in\mathbb D^n}\ \max_{1\le i\le k}|g_i(z)|=\delta>0.
\]
Then there exist constants \(c_1,\dots,c_k\in\mathbb C\) with
\[
\max_{1\le i\le k}|c_i|=\delta
\]
such that
\[
\inf_{z\in\mathbb D^n}\ \max_{1\le i\le k}|g_i(z)-c_i|=0.
\]
Thus the corresponding B\'ezout equation
\[
(g_1-c_1)f_1+\cdots+(g_k-c_k)f_k = 1
\]
has no solution.
\end{remark}

\section{Proof of Theorem~\ref{te2.1}}
\begin{proof}
We argue by induction on the number of variables \(n\).

For \(n=1\) the statement is Theorem~8.3 of \cite{Brudnyi2023}. Assume the result holds for dimension \(n\), and consider \(H^\infty(\mathbb{D}^{\,n+1},A)\).

Let \((\widehat U_i)_{1\le i\le k}\) be an open cover of \([M(H^\infty(\mathbb{D}))]^{\,n+1}\). By the product topology there exist a finite open cover \((\widehat V_j)_{1\le j\le l}\) of \([M(H^\infty(\mathbb{D}))]^n\) and, for each \(j\), a finite open cover \((\widehat W_{jm})_{1\le m\le l_j}\) of \(M(H^\infty(\mathbb{D}))\) such that the cover 
\((\widehat V_j\times \widehat W_{jm})_{j,m}\) of  \([M(H^\infty(\mathbb{D}))]^{\,n+1}\) refines \((\widehat U_i)_{1\le i\le k}\). Set
\[
U_{jm} := V_j\times W_{jm},\qquad 
V_j = \mathbb{D}^n\cap \widehat V_j,\quad 
W_{jm}=\mathbb{D}\cap \widehat W_{jm}.
\]
Then \((U_{jm})_{j,m}\) is a finite cover of \(\mathbb{D}^{\,n+1}\) refining \((U_i)_{1\le i\le k}\).

Suppose the Bézout equation
\begin{equation}\label{e3.1}
\sum_{i=1}^k f_i g_i = 1
\end{equation}
admits  local   solutions
\((g_{1;jm},\dots, g_{k;jm})\in [H^\infty(U_{jm},A)]^k\) on the cover  \((U_{jm})_{j,m}\).
For each \(i\in\{1,\dots, k\}\) define slice functions on  \(W_{jm}\),
\[
g_{i;jm}^s(z)(z_1,\dots,z_n):=g_{i;jm}(z_1,\dots,z_n,z),\qquad (z_1,\dots, z_n)\in V_j,\ z\in W_{jm}.
\]
These belong to \(H^\infty(W_{jm},H^\infty(V_j,A))\). Likewise, each \(f_i\) defines
\[
f_{i;j}^s \in H^\infty(\mathbb D,H^\infty(V_j,A)).
\]
Passing to slice functions in \eqref{e3.1} yields that the \(H^\infty(V_j,A)\)-valued Bézout equation on \(\mathbb D\),
\begin{equation}\label{e3.2}
\sum_{i=1}^k f_{i;j}^s \,g_i=1,
\end{equation}
has local  solutions \((g_{1;jm}^s,\dots, g_{k;jm}^s)\in [H^\infty(W_{jm},H^\infty(V_j,A))]^k\) on the cover \((W_{jm})\) of \(\mathbb D\). 

Since \(W_{jm}=\mathbb{D}\cap \widehat W_{jm}\) and \(\mathcal W_j=(\widehat W_{jm})_{1\le m\le l_j}\) is an open cover of \(M(H^\infty(\mathbb{D}))\), applying \cite[Thm.~8.3]{Brudnyi2023} we obtain global  solutions \((h_{1;j},\dots, h_{k;j})\in [H^\infty(\mathbb D,H^\infty(V_j,A))]^k\), \(1\le i\le k\), of \eqref{e3.2}.

Define
\[
g_{i;j}(z_1,\dots,z_{n+1}) := (h_{i;j}(z_{n+1}))(z_1,\dots,z_n),\qquad (z_1,\dots, z_{n+1})\in V_j\times\mathbb D.
\]
Then \(g_{i;j}\in H^\infty(V_j\times\mathbb D,A)\) and
\[
\sum_{i=1}^k f_i g_{i;j}=1 \quad\text{on } V_j\times \mathbb D.
\]
Hence we obtain local bounded holomorphic solutions of the original Bézout equation \eqref{e3.1}
on the cover \((V_j\times\mathbb D)_{1\le j\le l}\) of \(\mathbb D^{\,n+1}\).

Now view each \(g_{i;j}\) as a slice function
\[
g_{i;j}^s \in H^\infty(V_j,H^\infty(\mathbb D,A)).
\]
Similarly, \(f_i\) induce slice functions
\[
f_i^s \in H^\infty(\mathbb D^n,H^\infty(\mathbb D,A)).
\]
Then the solvability of the Bézout equation \eqref{e3.1} on the cover \((V_j\times \mathbb D)\) transforms into the identities
\[
\sum_{i=1}^k f_i^s g_{i;j}^s = 1 \quad \text{in } H^\infty(\mathbb D,A)\ \text{for each } V_j,
\]
i.e. we obtain local bounded holomorphic solutions of the Bézout equation on \(\mathbb D^n\),
\begin{equation}\label{e3.3}
\sum_{i=1}^k f_i^s h_i = 1
\end{equation}
with values in the algebra \(H^\infty(\mathbb D,A)\), over the cover \((V_j)_{1\le j\le k}\) of \(\mathbb D^n\), where \(V_j = \mathbb{D}^n\cap \widehat V_j\) and \(\mathcal V=(\widehat V_j)_{1\le j\le l}\) is a finite open cover of \([M(H^\infty(\mathbb{D}))]^n\).

By the induction hypothesis, there exists a solution \((h_1,\dots, h_k) \in [H^\infty(\mathbb D^n,H^\infty(\mathbb D,A))]^k\) of \eqref{e3.3}. 

Finally, set
\[
g_i(z_1,\dots,z_{n+1}) := (h_i(z_1,\dots,z_n))(z_{n+1}), \qquad (z_1,\dots, z_{n+1})\in\mathbb D^{n+1}.
\]
Then \(g_i\in H^\infty(\mathbb D^{n+1},A)\) and 
\[
\sum_{i=1}^k f_i g_i=1.
\]
This completes the induction.
\end{proof}

\section{ Proof of Theorem~\ref{teo2.2}}
\begin{proof}
Recall that \(A(F_\xi)\) is the uniform closure of \(\widehat{H^\infty(\mathbb{D}^n)}|_{F_\xi}\), hence \(M\!\left(A(F_\xi)\right)\) is exactly the \(H^\infty\)-hull of \(F_\xi\). Also recall that
\[
\pi_n \colon M\!\left(H^\infty(\mathbb{D}^n)\right)\longrightarrow \bigl[M(H^\infty(\mathbb{D}))\bigr]^{n}
\]
is the transpose of the canonical embedding
\[
\widehat{\otimes}_{\varepsilon}^{\,n} H^\infty(\mathbb{D})\ \hookrightarrow\ H^\infty(\mathbb{D}^n),
\]
and that the embedded algebra separates the fibers of \(\pi_n\).

Since the \(H^\infty\)-hull of \(\pi_n^{-1}(\xi)\) coincides with \(\pi_n^{-1}(\xi)\), we have \(M\!\left(A(F_\xi)\right)\subset \pi_n^{-1}(\xi)\).

For the reverse inclusion, suppose \(\zeta\in \pi_n^{-1}(\xi)\setminus M\!\left(A(F_\xi)\right)\). By the definition of the \(H^\infty\)-hull, there exists \(f\in H^\infty(\mathbb{D}^n)\) with \(\hat f(\zeta)=1\) and
\[
m:=\max_{F_\xi}|\hat f|<1.
\]
Set \(g:=1-f\). Then \(\hat g(\zeta)=0\) and \(|\hat g(x)|\ge 1-m>0\) for all \(x\in F_\xi\). By continuity of \(\hat g\) and compactness of \(F_\xi\), there exists an open neighborhood \(U\subset \bigl[M(H^\infty(\mathbb{D}))\bigr]^n\) of \(\zeta\) such that
\begin{equation}\label{e4.1}
|\hat g(y)|\ \ge\ \tfrac12(1-m)\qquad \text{for all } y\in \pi_n^{-1}(U).
\end{equation}

By the product topology on \(\bigl[M(H^\infty(\mathbb{D}))\bigr]^n\), there exist functions
\[
g_1,\dots,g_k\ \in\ \widehat{\otimes}_{\varepsilon}^{\,n} H^\infty(\mathbb{D})
\]
whose Gelfand transforms satisfy, for some \(\delta>0\) and for a neighborhood \(W\subset U\) of \(\zeta\) with \(\overline{W}\subset U\),
\begin{equation}\label{e4.2}
\hat g_i|_{F_\xi}=0\ \ (1\le i\le k),\qquad
\max_{1\le i\le k}|\hat g_i(y)|\ \ge\ \delta\ \ \text{for all }y\in \pi_n^{-1}(W).
\end{equation}
Since each \(g_i\) extends continuously to \(\bigl[M(H^\infty(\mathbb{D}))\bigr]^n\), write \(h_i\) for that extension. Then \(\hat g_i=h_i\circ \pi_n\), so \(\hat g_i\) is constant on the fibers of \(\pi_n\). In particular, from \eqref{e4.2} it follows that for every \(x\in \bigl[M(H^\infty(\mathbb{D}))\bigr]^n\setminus W\) there exist an open neighborhood \(W_x\) and an index \(i=i(x)\) with
\begin{equation}\label{e4.3}
|\hat g_i(y)|\ \ge\ \delta_x>0\qquad \text{for all } y\in \pi_n^{-1}(W_x).
\end{equation}
By compactness, there is a finite subcover \((W_{x_j})_{j=1}^{s}\) of \(\bigl[M(H^\infty(\mathbb{D}))\bigr]^n\setminus W\).

Thus the tuple \((g,g_1,\dots,g_k)\subset H^\infty(\mathbb{D}^n)\) satisfies a local corona condition on the finite open cover
\[
\bigl(W\cap \mathbb{D}^n,\ W_{x_1}\cap \mathbb{D}^n,\ \dots,\ W_{x_s}\cap \mathbb{D}^n\bigr)
\]
of \(\mathbb{D}^n\). Indeed, on \(W \cap \mathbb{D}^n\), inequality \eqref{e4.1} gives a uniform lower bound for \(|g|\), and on each \(W_{x_j} \cap \mathbb{D}^n\) inequality \eqref{e4.3} gives a uniform lower bound for \(|g_{i(x_j)}|\).
By Theorem~\ref{te2.1} (recall it asserts that local solvability of a Bézout equation on a
finite open cover of \(\mathbb{D}^n\), arising from a cover of \(\bigl[M(H^\infty(\mathbb{D}))\bigr]^n\) via \(\pi_n\),
yields a global solution in \(H^\infty(\mathbb{D}^n)\)),
 there exist \(f_0,f_1,\dots,f_k\in H^\infty(\mathbb{D}^n)\) such that
\[
g\,f_0+g_1 f_1+\cdots+g_k f_k=1\quad \text{on }\mathbb{D}^n.
\]

Taking Gelfand transforms and restricting to \(\pi_n^{-1}(\xi)\), and using \(\hat g_i|_{F_\xi}=0\) for \(1\le i\le k\), we obtain
\[
\hat f_0(x)\,\hat g(x)=1\qquad \text{for all }x\in \pi_n^{-1}(\xi).
\]
But \(\hat g(\zeta)=0\), a contradiction. Hence \(\pi_n^{-1}(\xi)\subset M\!\left(A(F_\xi)\right)\), and therefore
\[
\pi_n^{-1}(\xi)=M\!\left(A(F_\xi)\right).
\]

Finally,
\[
\mathrm{cl}(\mathbb{D}^n)=M\!\left(H^\infty(\mathbb{D}^n)\right)
\]
(i.e., the corona theorem holds for \(H^\infty(\mathbb{D}^n)\)) if and only if
\[
M\!\left(A(F_\xi)\right)=F_\xi\qquad \text{for all }\xi\in \bigl[M(H^\infty(\mathbb{D}))\bigr]^n,
\]
equivalently, the \(H^\infty\)-hull of each fiber \(F_\xi\) coincides with \(F_\xi\).
\end{proof}
\section{Proof of Theorem~\ref{teo2.3}}
\begin{proof}[Proof of (i)]
Let \( \xi = (\xi_1, \dots, \xi_n) \in M_a^n \). 
Since each coordinate \( \xi_j \in  M(H^\infty(\mathbb{D})) \) lies in a nontrivial Gleason part, a classical result of Hoffman 
(see, e.g., \cite[Thm.~5.5]{Hoffman1967}) guarantees the existence of an interpolating sequence
\[
\zeta_j = \{ z_{j,k} \}_{k=1}^\infty \subset \mathbb{D}
\]
whose closure in \( M(H^\infty(\mathbb{D})) \) contains \( \xi_j \), for \( j = 1, \dots, n \).

\smallskip

Define the product sequence
\[
\zeta := \zeta_1 \times \cdots \times \zeta_n 
= \left\{ z_k := (z_{1,k}, \dots, z_{n,k}) \right\}_{k=1}^\infty \subset \mathbb{D}^n,
\]
whose closure in the product topology on \( [M(H^\infty(\mathbb{D}))]^n \) contains \( \xi \).\smallskip

\noindent \textsl{CLAIM 1.} 
\textit{The sequence \( \zeta \) is interpolating for \( H^\infty(\mathbb{D}^n) \).}

\smallskip

Indeed, for each  \( j \) and \( k \), let \( \varphi_{j,k} \in H^\infty(\mathbb{D}) \) be bounded holomorphic functions satisfying
\[
\varphi_{j,k}(z_{j,\ell}) = \delta_{k,\ell}, \qquad 
\sum_{k=1}^\infty |\varphi_{j,k}(z)| \le C_j, \quad \text{for all } z \in \mathbb{D},
\]
for some constants \( C_j = C_j(\zeta_j) > 0 \); see \cite[Ch.~VII, Thm.~2.1]{Garnett1981}. Define
\[
\psi_k(z) := \prod_{j=1}^n \varphi_{j,k}(z_j), 
\quad z = (z_1, \dots, z_n) \in \mathbb{D}^n, \quad k \in \mathbb{N}.
\]
Then \( \psi_k \in H^\infty(\mathbb{D}^n) \), and
\[
\psi_k(z_\ell) = \delta_{k,\ell}, \quad \forall\, k, \ell \in \mathbb{N}, \qquad 
\sum_{k=1}^\infty |\psi_k(z)| \le C := \prod_{j=1}^n C_j, \quad \forall\, z \in \mathbb{D}^n.
\]
Thus, the  map
\[
L \colon \ell^\infty(\zeta) \to H^\infty(\mathbb{D}^n), \qquad 
L\big( (c_k)_{k=1}^\infty \big) := \sum_{k=1}^\infty c_k \psi_k,
\]
is a bounded linear  operator of norm \( \le C \), with 
\( L((c_k))|_\zeta = (c_k) \) for all \( (c_k) \in \ell^\infty(\zeta) \). 
Hence \( \zeta \) is interpolating for \(H^\infty(\mathbb{D}^n)\), as claimed.

\smallskip

 It follows from \textsl{Claim 1} that the restriction map
\[
H^\infty(\mathbb{D}^n) \to \ell^\infty(\zeta), \qquad f \mapsto f|_\zeta,
\]
induces, by transposition, an embedding of maximal ideal spaces
\[
i_\zeta \colon M\big( \ell^\infty(\zeta) \big) \hookrightarrow M\big( H^\infty(\mathbb{D}^n) \big).
\]
Since \( \ell^\infty(\zeta) \) is self-adjoint, its maximal ideal space 
\( M(\ell^\infty(\zeta)) \) is homeomorphic to the Stone–Čech compactification \( \beta(\mathbb{N}) \), and
\[
i_\zeta \big( M(\ell^\infty(\zeta)) \big) 
= \operatorname{cl}(\zeta) \subset \operatorname{cl}(\mathbb{D}^n)
\]
is totally disconnected.

\smallskip

Let \( B_{\zeta_j} \in H^\infty(\mathbb{D}) \) be the Blaschke product vanishing precisely on \( \zeta_j \), 
and denote by
\( \hat{B}_{\zeta_j} \in C\big( M(H^\infty(\mathbb{D})) \big)\)
its Gelfand transform. 
Let \( p_j \colon [M(H^\infty(\mathbb{D}))]^n \to M(H^\infty(\mathbb{D})) \) 
be the projection onto the \( j \)-th coordinate, and define
\[
\widetilde{B}_{\zeta_j} := \hat{B}_{\zeta_j} \circ p_j 
\in C\big( [M(H^\infty(\mathbb{D}))]^n \big).
\]
By the product topology, the joint zero set of the functions \( \widetilde{B}_{\zeta_j} \), \( 1 \le j \le n \), 
coincides with the closure of the interpolating sequence \( \zeta \) in \( [M(H^\infty(\mathbb{D}))]^n \) 
and hence contains \( \xi \).

By definition, the pullbacks \( \pi_n^* \widetilde{B}_{\zeta_j} \in C\big(M(H^\infty(\mathbb{D}^n))\big) \) 
are the Gelfand transforms of the functions \( (p_j|_{\mathbb{D}^n})^*(B_{\zeta_j}) \in H^\infty(\mathbb{D}^n) \), given by
\[
(p_j|_{\mathbb{D}^n})^*(B_{\zeta_j})(z_1,\dots, z_n) = B_{\zeta_j}(z_j), 
\qquad 1 \le j \le n,
\]
and they vanish on \(F_\xi\).

We next identify the joint zero set of these pullbacks.
\smallskip

\noindent \textsl{CLAIM 2.} 
\textit{The joint zero set of the functions \( \pi_n^* \widetilde{B}_{\zeta_j} \) 
on \( \operatorname{cl}(\mathbb{D}^n) \subset M(H^\infty(\mathbb{D}^n)) \), \( 1 \le j \le n \), 
is precisely \( \operatorname{cl}(\zeta) \).}

\smallskip

To prove this, suppose,  to the contrary, that there exists a point 
\( \eta \in \operatorname{cl}(\mathbb{D}^n) \setminus \operatorname{cl}(\zeta) \) such that
\[
\pi_n^* \widetilde{B}_{\zeta_j}(\eta) = 0 
\quad \text{for all } j = 1, \dots, n.
\]
Let \( U \subset M(H^\infty(\mathbb{D}^n)) \) be an open neighborhood of \( \eta \) whose closure is disjoint from \( \operatorname{cl}(\zeta) \). 
For each \( \varepsilon \in (0,1) \), consider the open set
\[
U_\varepsilon := 
U \cap \bigl\{ \theta \in M(H^\infty(\mathbb{D}^n)) : 
|\pi_n^* \widetilde{B}_{\zeta_j}(\theta)| < \varepsilon 
\text{ for all } 1 \le j \le n \bigr\}.
\]
Define the pseudohyperbolic metric on \( \mathbb{D}^n \) by
\[
\rho_n(z, w) := \max_{1 \le j \le n} \rho(z_j, w_j), \qquad z, w \in \mathbb{D}^n,
\]
where \( \rho \) is the pseudohyperbolic metric on \( \mathbb{D} \), see \eqref{e2.5}. 
Since \( \mathbb{D}^n \) is dense in \( \operatorname{cl}(\mathbb{D}^n) \), we may select a point
\[
z_\varepsilon \in U_\varepsilon \cap \mathbb{D}^n.
\]
By a result of Hoffman (see, e.g., \cite[Ch.~X, Lemma~1.4]{Garnett1981}), we have
\begin{equation}\label{e5.1}
\lim_{\varepsilon \to 0} \operatorname{dist}(z_\varepsilon, \zeta) = 0,
\end{equation}
where, for \( z \in \mathbb{D}^n \) and \( Z \subset \mathbb{D}^n \),
\[
\operatorname{dist}(z, Z) := \inf_{y \in Z} \rho_n(z, y).
\]
Let \( \nu \in \operatorname{cl}(\mathbb{D}^n) \) be a limit point of the sequence \( \{ z_{1/k} \}_{k \in \mathbb{N}} \). 
By the Schwarz–Pick lemma (see, e.g., \cite[Ch.~I, Lemma~1.2]{Garnett1981}) and \eqref{e5.1}, 
we obtain \( \nu \in \operatorname{cl}(\zeta) \).
On the other hand, since \( \nu \in \operatorname{cl}(U) \) and 
\( \operatorname{cl}(U) \cap \operatorname{cl}(\zeta) = \varnothing \), 
a contradiction, proving \textsl{Claim 2}.
\smallskip

\textsl{Claim 2} implies that
\[
F_\xi := \pi_n^{-1}(\xi) \cap \operatorname{cl}(\mathbb{D}^n) 
\subset 
\left\{ \nu\in \operatorname{cl}(\mathbb{D}^n) : 
\pi_n^* \widetilde{B}_{\zeta_j}(\nu) = 0,\ 1 \le j \le n \right\} 
\subset 
\operatorname{cl}(\zeta).
\]
Since \( \zeta \) is an interpolating sequence for \( H^\infty(\mathbb{D}^n) \), 
the restriction of the algebra 
\( \mathcal{O}\big(M(H^\infty(\mathbb{D}^n))\big) \) 
to \( \operatorname{cl}(\zeta) \) coincides with the full algebra 
\( C\big(\operatorname{cl}(\zeta)\big) \). 
By the Tietze–Urysohn extension theorem, the same property holds for any closed subset of \( \operatorname{cl}(\zeta) \), 
and in particular for the compact set \( F_\xi \subset \operatorname{cl}(\zeta) \). 
Hence \( A(F_\xi) = C(F_\xi) \), and therefore the maximal ideal space \( M(A(F_\xi)) \) coincides with \( F_\xi \). 
By Theorem~\ref{teo2.2}, we conclude that
\[
\pi_n^{-1}(\xi) = F_\xi \subset \operatorname{cl}(\zeta).
\]
As \( \operatorname{cl}(\zeta) \) is totally disconnected, the same is true for \( F_\xi \).

This completes the proof of part~(i) of the theorem.\medskip

\noindent\textit{Proof of (ii).} 
If exactly one coordinate of \( \xi \) lies in \( M_s \) and the remaining \( n-1 \) coordinates lie in \( M_a \), 
then the natural action of the symmetric group on coordinates of \( \mathbb{D}^n \) 
preserves both algebras \( \widehat{\otimes}_\varepsilon^n H^\infty(\mathbb{D}) \) and \( H^\infty(\mathbb{D}^n) \). 
Consequently, it extends to actions on their maximal ideal spaces 
\( \bigl[M\big(H^\infty(\mathbb{D})\big)\bigr]^n\) and \( M\big(H^\infty(\mathbb{D}^n)\big) \), 
and these actions commute with the projection
\[
\pi_n \colon M\big(H^\infty(\mathbb{D}^n)\big) \to 
\bigl[M\big(H^\infty(\mathbb{D})\big)\bigr]^n.
\]
In particular, the fibers of \( \pi_n \) corresponding to points in the same orbit 
are homeomorphic and have identical analytic structure. 
Hence, without loss of generality, we may and will assume that 
\[
\xi \in M_a^{n-1} \times M_s.
\]

We first show that
\[
\pi_n^{-1}(\xi)
=
F_\xi := \pi_n^{-1}(\xi)\cap \mathrm{cl}(\mathbb{D}^n),
\]
see Lemma~\ref{lem5.1} below.

Let \(\xi = (\xi^{\,n-1},\xi_n)\), where \(\xi^{\,n-1} \in M_a^{\,n-1}\) and \(\xi_n \in M_s\). By \textsl{Claim~1} in the proof of part~(i) of the theorem, there exists an interpolating sequence \(\zeta = \{z_k\} \subset \mathbb{D}^{\,n-1}\) for \(H^\infty(\mathbb{D}^{\,n-1})\), which is the Cartesian product of interpolating sequences for \(H^\infty(\mathbb{D})\) in the coordinate disks, such that the closure of \(\zeta\) in \(\bigl[M(H^\infty(\mathbb{D}))\bigr]^{n-1}\) contains the point \(\xi^{\,n-1}\).
\noindent
Moreover, there exist functions \(\psi_k \in H^\infty(\mathbb{D}^{\,n-1})\), \(k \in \mathbb{N}\), such that
\[
\psi_k(z_i) = \delta_{k,i}, \quad \forall\, k,i \in \mathbb{N}, \qquad
\sum_{k=1}^{\infty} \lvert \psi_k(z) \rvert \le C = C(\zeta), \quad \forall\, z \in \mathbb{D}^{\,n-1}.
\]
\noindent
Then for every uniformly bounded sequence of functions \(\{f_k\} \subset H^\infty(\mathbb{D})\), there exists \(f \in H^\infty(\mathbb{D}^n)\) such that
\[
f(z_k, z) = f_k(z) \quad \forall\, z \in \mathbb{D},\ \forall\, k \in \mathbb{N}.
\]
Indeed, set
\[
f(w,z) := \sum_{k=1}^{\infty} \psi_k(w)\, f_k(z), \qquad (w,z) \in \mathbb{D}^{\,n-1} \times \mathbb{D}.
\]
Since \(\sum_{k=1}^{\infty} \lvert \psi_k(w) \rvert \le C\) for all \(w \in \mathbb{D}^{\,n-1}\) and \(\sup_k \lVert f_k \rVert_\infty < \infty\), the series defining \(f\) converges uniformly on compact subsets of \(\mathbb{D}^{\,n-1} \times \mathbb{D}\), and
\[
\lVert f \rVert_\infty \le C\, \sup_{k} \lVert f_k \rVert_\infty.
\]
Uniform convergence on compacta shows that \(f \in H^\infty(\mathbb{D}^n)\). Moreover, for each \(i \in \mathbb{N}\) and every \(z \in \mathbb{D}\),
\[
f(z_i, z) = \sum_{k=1}^{\infty} \psi_k(z_i)\, f_k(z) = \sum_{k=1}^{\infty} \delta_{k,i}\, f_k(z) = f_i(z),
\]
as required.

Thus the restriction map
\begin{equation}\label{e5.2}
R \colon H^\infty(\mathbb{D}^n) \to H^\infty(\mathbb{D}\times\mathbb{N}), 
\qquad
R(f)(k) := f \bigl|_{\{z_k\}\times \mathbb{D}},
\end{equation}
is a surjective homomorphism of Banach algebras, where \(H^\infty(\mathbb{D}\times\mathbb{N})\) is equipped with the norm
\[
\lVert (g_k)_{k\in\mathbb{N}} \rVert
:= \sup_{k\in\mathbb{N}} \lVert g_k \rVert_\infty.
\]
\smallskip
\noindent
\noindent
In particular, the induced transpose map on maximal ideal spaces
\[
R^* \colon M\bigl(H^\infty(\mathbb{D}\times\mathbb{N})\bigr) \longrightarrow M\bigl(H^\infty(\mathbb{D}^n)\bigr),
\qquad
R^*(\varphi) := \varphi \circ R,
\]
is a continuous embedding, and its image is \(H^\infty\)\nobreakdash-convex; that is, for every character
\(\eta \in M\bigl(H^\infty(\mathbb{D}^n)\bigr) \setminus R^*\!\bigl(M(H^\infty(\mathbb{D}\times\mathbb{N}))\bigr)\)
there exists \(f \in H^\infty(\mathbb{D}^n)\) such that
\[
\lvert \hat{f}(\eta) \rvert
>
\sup_{ R^*(M(H^\infty(\mathbb{D}\times\mathbb{N}))) }
\lvert \hat{f} \rvert .
\]

\begin{lemma}\label{lem5.1}
We have
\[
R^*\bigl(M(H^\infty(\mathbb{D}\times\mathbb{N}))\bigr)
\supset
\pi_n^{-1}(\xi)
=
F_\xi.
\]
\end{lemma}

\begin{proof}
Let us prove first that \(\pi_n^{-1}(\xi) \subset \operatorname{im} R^*\).
Set
\[
I := \ker R
=
\bigl\{ f \in H^\infty(\mathbb{D}^n) : f(z_k, z) = 0 \ \forall\,k\in\mathbb{N},\ \forall z\in\mathbb{D} \bigr\}.
\]
Let \(U\) be an open neighborhood of
\[
F_\xi := \pi_n^{-1}(\xi) \cap \mathrm{cl}(\mathbb{D}^n)
\]
in \(M\bigl(H^\infty(\mathbb{D}^n)\bigr)\). As observed above, we may assume without loss of generality that
\[
U=\pi_n^{-1}(V)\cap \mathrm{cl}(\mathbb{D}^n)
\]
for some open neighborhood \(V\) of \(\xi\) in \(\bigl[M(H^\infty(\mathbb{D}))\bigr]^n\).
Since the closure of \(\zeta \times \mathbb{D}\) in
\(\bigl[M(H^\infty(\mathbb{D}))\bigr]^n\) contains \(\xi\),
there exists a subsequence \(\zeta_U\subset \zeta\) such that
\(\zeta_U\times \mathbb{D}\subset V\).
Varying \(U\) over a basis of open neighborhoods of \(F_\xi\), we conclude that 
\[
F_\xi
\subset
\mathrm{cl}(\zeta\times\mathbb{D})
\qquad
\text{(recall that \(\mathrm{cl} = \mathrm{cl}_{M(H^\infty(\mathbb{D}^n))}\))}.
\]

Thus, the Gelfand transform of every \(f\in I\) vanishes on \(F_\xi\), and hence on its
\(H^\infty\)\nobreakdash-hull. By Theorem~\ref{teo2.2} this hull equals \(\pi_n^{-1}(\xi)\).
Therefore every \(\eta\in\pi_n^{-1}(\xi)\) annihilates \(I\) and thus factors through
\[
H^\infty(\mathbb{D}^n)/I\simeq H^\infty(\mathbb{D}\times\mathbb{N}),
\]
i.e.,
\[
\eta=R^*(\varphi)
\qquad\text{for some }\varphi\in M\!\bigl(H^\infty(\mathbb{D}\times\mathbb{N})\bigr).
\]
This proves \(\pi_n^{-1}(\xi)\subset \operatorname{im}R^*\).
Using this fact, let us prove that
\(\pi_n^{-1}(\xi)=F_\xi\).
\smallskip

Indeed, the Corona Theorem for \(H^\infty(\mathbb{D}\times\mathbb{N})\) 
(which follows from uniform bounds in the proof of Carleson's original theorem, see, e.g., \cite{Garnett1981})
implies that \(\mathbb{D}\times\mathbb{N}\) is an open dense subset of
\(M(H^\infty(\mathbb{D}\times\mathbb{N}))\). By construction of \(R^*\),
\[
R^*(\mathbb{D}\times\mathbb{N}) = \zeta\times\mathbb{D}.
\]
Since \(R^*\) is continuous and injective, it follows that
\[
R^*\bigl(M(H^\infty(\mathbb{D}\times\mathbb{N}))\bigr)
=
\mathrm{cl}\bigl(\zeta\times\mathbb{D}\bigr)
\subset \mathrm{cl}(\mathbb{D}^n).
\]
Hence
\[
\pi_n^{-1}(\xi)
=
\pi_n^{-1}(\xi)\cap
R^*\bigl(M(H^\infty(\mathbb{D}\times\mathbb{N}))\bigr)
\subset
\pi_n^{-1}(\xi)\cap \mathrm{cl}(\mathbb{D}^n)
=
F_\xi.
\]
Thus, \(\pi_n^{-1}(\xi)=F_\xi\), completing the proof of the lemma.
\end{proof}

By Lemma~\ref{lem5.1} and Theorem~\ref{teo2.2} we conclude that
\(M(A(F_\xi)) = F_\xi\).
\medskip

We now use the known structure of 
\(M\bigl(H^\infty(\mathbb{D}\times\mathbb{N})\bigr)\)
to analyze the internal geometry of \(F_\xi\).
As shown in \cite{Brudnyi2021}, every Gleason part in 
\(M\bigl(H^\infty(\mathbb{D}\times\mathbb{N})\bigr)\)
is either a singleton or an analytic disk; the set of nontrivial parts is open and dense,
the set of trivial parts is totally disconnected, and
\[
\dim M\bigl(H^\infty(\mathbb{D}\times\mathbb{N})\bigr)=2.
\]
We  relate this structure to the fibers of the projection \(T^*\) introduced below.

\smallskip 
The Banach algebra homomorphism
\begin{equation}\label{e5.3}
T \colon \ell^\infty \longrightarrow H^\infty(\mathbb{D}\times\mathbb{N}), \qquad
T\bigl((c_k)_{k\in\mathbb{N}}\bigr) := (g_k)_{k\in\mathbb{N}}, \quad g_k(z) \equiv c_k,\ z\in\mathbb{D},
\end{equation}
is an isometric embedding. Identifying \(M(\ell^\infty)\) with the Stone–Čech compactification \(\beta\mathbb{N}\), its transpose
\[
T^* \colon M\bigl(H^\infty(\mathbb{D}\times\mathbb{N})\bigr) \longrightarrow \beta\mathbb{N}
\]
is surjective. For characters corresponding to points \((z,n)\in \mathbb{D}\times\mathbb{N}\),
the map \(T^*\) assigns evaluation at the index \(n\); in particular,
\[
T^*\bigl(\mathbb{D}\times\mathbb{N}\bigr)=\mathbb{N}\subset \beta\mathbb{N}.
\]

\smallskip\noindent
Fix \(\chi \in \beta\mathbb{N}\setminus\mathbb{N}\).
If \(G \subset M \bigl(H^\infty(\mathbb{D}\times\mathbb{N})\bigr)\) is an analytic disk with \(T^*(G)\) containing \(\chi\), then, since \(T^*\) is continuous and
\(\beta\mathbb{N}\) is totally disconnected, the connected set \(T^*(G)\) must be a singleton. Hence \(T^*(G)=\{\chi\}\), so \(G \subset (T^*)^{-1}(\chi)\).
Consequently, for every \(\chi \in \beta\mathbb{N}\setminus\mathbb{N}\), the fiber \((T^*)^{-1}(\chi)\) is a disjoint union of analytic disks and points belonging to trivial Gleason parts of \(M \bigl(H^\infty(\mathbb{D}\times\mathbb{N})\bigr)\) (each such part is a singleton).

\begin{lemma}\label{lem5.2}
Let \(G\) be an analytic disk in \(M\bigl(H^\infty(\mathbb{D}\times\mathbb{N})\bigr)\) such that
\(R^*(G)\) contains a point \(\eta \in F_{\xi} = \pi_n^{-1}(\xi)\),
where \(\xi = (\xi^{\,n-1},\xi_n) \in M_a^{n-1}\times M_s \subset \bigl[M(H^\infty(\mathbb{D}))\bigr]^n\).
Then
\[
R^*(G) \subset  F_\xi.
\]
\end{lemma}

\begin{proof}
Let 
\[
p^{\,n-1}\colon \mathbb{D}^n \to \mathbb{D}^{\,n-1}
\]
be the projection onto the first \(n-1\) coordinates. It induces the following algebra homomorphisms.
First, the pullback
\begin{equation}\label{e5.4}
(p^{\,n-1})^*\colon H^\infty(\mathbb{D}^{\,n-1})
\longrightarrow H^\infty(\mathbb{D}^n),\qquad
(p^{\,n-1})^*(f)=f\circ p^{\,n-1},
\end{equation}
whose transpose is the map
\[
\hat{p}^{\,n-1}\colon M\bigl(H^\infty(\mathbb{D}^n)\bigr)
\longrightarrow M\bigl(H^\infty(\mathbb{D}^{\,n-1})\bigr),\qquad
\hat{p}^{\,n-1}(\varphi)=\varphi\circ (p^{\,n-1})^*,
\]
and \(\hat{p}^{\,n-1}|_{\mathbb{D}^n}=p^{\,n-1}\).

 Second, the restriction
\[
(p^{\,n-1})^*\big|_{\widehat{\otimes}_\varepsilon^{\,n-1} H^\infty(\mathbb{D})}
\;\colon\;
\widehat{\otimes}_\varepsilon^{\,n-1} H^\infty(\mathbb{D})
\longrightarrow
\widehat{\otimes}_\varepsilon^{\,n} H^\infty(\mathbb{D})
\]
whose transpose
\[
\bar{p}^{\,n-1}\colon
\bigl[M\bigl(H^\infty(\mathbb{D})\bigr)\bigr]^n
\longrightarrow
\bigl[M\bigl(H^\infty(\mathbb{D})\bigr)\bigr]^{\,n-1},
\]
coincides with the projection onto the first \(n-1\) coordinates.

The embeddings
\(\widehat{\otimes}_\varepsilon^{\,k} H^\infty(\mathbb{D})
\hookrightarrow H^\infty(\mathbb{D}^k)\)
are compatible with the pullback \((p^{\,n-1})^*\). Passing to transposes (see \eqref{e2.2}) yields
\begin{equation}\label{e5.5}
\bar{p}^{\,n-1}\circ \pi_{n}
\;=\;
\pi_{\,n-1}\circ \hat{p}^{\,n-1}.
\end{equation}

Let \(\zeta = \{z_k\} \subset \mathbb{D}^{\,n-1}\) be the interpolating sequence fixed at the beginning of the proof, whose closure in \(\bigl[M(H^\infty(\mathbb{D}))\bigr]^{n-1}\) contains \(\xi^{\,n-1}\). Define the restriction homomorphism
\[
R_\zeta\colon H^\infty(\mathbb{D}^{\,n-1}) \longrightarrow \ell^\infty(\mathbb{N}),\qquad
R_\zeta(f)=\bigl(f(z_k)\bigr)_{k\in\mathbb{N}}.
\]

By the definitions of \(R\), \(T\), and \((p^{\,n-1})^*\), see \eqref{e5.2}--\eqref{e5.4}, we have the commutative relation
\[
T\circ R_\zeta \;=\; R\circ (p^{\,n-1})^* .
\]
Passing to transposes and combining with \eqref{e5.5} gives the commutative diagram:
\begin{equation}\label{e5.6}
\begin{tikzcd}[column sep=8em, row sep=3.5em]
M\bigl(H^\infty(\mathbb{D}\times\mathbb{N})\bigr)
  \arrow[r, "T^*"]
  \arrow[d, "R^*"']
&
\beta\mathbb{N}
  \arrow[d, "(R_\zeta)^*"]
\\
M\bigl(H^\infty(\mathbb{D}^n)\bigr)
  \arrow[r, "\hat{p}^{\,n-1}"']
  \arrow[d, "\pi_n"']
&
M\bigl(H^\infty(\mathbb{D}^{\,n-1})\bigr)
  \arrow[d, "\pi_{\,n-1}"]
\\
\bigl[M\bigl(H^\infty(\mathbb{D})\bigr)\bigr]^n
  \arrow[r, "\bar{p}^{\,n-1}"']
&
\bigl[M\bigl(H^\infty(\mathbb{D})\bigr)\bigr]^{\,n-1}
\end{tikzcd}
\end{equation}

Since \(R^*(G)\subset M\bigl(H^\infty(\mathbb{D}^n)\bigr)\) meets the fibre \(F_\xi=\pi_n^{-1}(\xi)\) and \(T^*\) is constant on \(G\), the diagram \eqref{e5.6} gives, for \(\eta\in R^*(G)\cap F_\xi\),
\begin{equation}\label{e5.7}
(\pi_{\,n-1}\circ \hat{p}^{\,n-1})(\eta)=\xi^{\,n-1}.
\end{equation}
Let \(T^*(G)=\{\omega\}\subset\beta\mathbb N\). Then the top square in \eqref{e5.6} yields
\[
\hat{p}^{\,n-1}\bigl(R^*(G)\bigr)=\{(R_\zeta)^*(\omega)\}.
\]
Applying \(\pi_{\,n-1}\) and using \eqref{e5.7}, we obtain
\[
(\pi_{\,n-1}\circ \hat{p}^{\,n-1})\bigl(R^*(G)\bigr)=\{\xi^{\,n-1}\}.
\]

For the last coordinate, let
\[
p_n:\bigl[M(H^\infty(\mathbb{D}))\bigr]^n\longrightarrow M(H^\infty(\mathbb{D}))
\]
be the projection onto the \(n\)-th factor. Then \((p_n\circ \pi_n)\bigl(R^*(G)\bigr)\) is an analytic disk in \(M(H^\infty(\mathbb{D}))\) passing through \(\xi_n\). Since \(\xi_n\in M_s\) is a trivial Gleason part, Hoffman’s results \cite[Thm.~4.3, 5.5, 5.6]{Hoffman1967} imply that this image is a singleton:
\[
(p_n\circ \pi_n)\bigl(R^*(G)\bigr)=\{\xi_n\}.
\]

Combining the two displays gives
\[
\pi_n\bigl(R^*(G)\bigr)=\{(\xi^{\,n-1},\xi_n)\}=\{\xi\},
\]
hence \(R^*(G)\subset \pi_n^{-1}(\xi)=F_\xi\), as claimed.
\end{proof}

Combining Lemmas~\ref{lem5.1}--\ref{lem5.2}
with the structure of $M\bigl(H^\infty(\mathbb{D}\times\mathbb{N})\bigr)$
established in \cite{Brudnyi2021}, we obtain that
$F_\xi$ is a disjoint union of its Gleason parts, each of which is either
a single point or an analytic disk.
Moreover, since $F_\xi$ embeds into 
$M\bigl(H^\infty(\mathbb{D}\times\mathbb{N})\bigr)$,
its covering dimension satisfies
\[
\dim F_\xi \le 2.
\]

To complete the proof of part~\textup{(ii)}, it remains to show that
the subset of trivial Gleason parts in \(F_\xi\) is nonempty.
Equivalently, if the fiber \(F_\xi\) contains a nontrivial Gleason part (an analytic disk),
then it must also contain at least one trivial Gleason part.

\begin{lemma}\label{lem5.3}
The closure in \(M\bigl(H^\infty(\mathbb{D}\times\mathbb{N})\bigr)\)
of any nontrivial Gleason part of \(H^\infty(\mathbb{D}\times\mathbb{N})\)
contains a trivial Gleason part of this algebra.
\end{lemma}

\begin{proof}
It was shown in~\cite[Lemma~2]{Be} that the algebra 
\(H^\infty(\mathbb{D}\times\mathbb{N})\) is logmodular on 
\(\mathbb{D} \times \mathbb{N}\).
This, together with the general result of 
K.~Hoffman~\cite{Hoffman1962}, implies that the Gleason parts of 
\(H^\infty(\mathbb{D}\times\mathbb{N})\) are either points or analytic disks.
The latter means that there exists a continuous injective map 
\[
\pi:\mathbb{D}\longrightarrow M\bigl(H^\infty(\mathbb{D}\times\mathbb{N})\bigr)
\]
onto this Gleason part such that 
\(\hat f\circ\pi \in H^\infty(\mathbb{D})\) 
for all \(f\in H^\infty(\mathbb{D}\times\mathbb{N})\).

Suppose, to the contrary, that there exists a nontrivial Gleason part 
\(G \subset M\bigl(H^\infty(\mathbb{D}\times\mathbb{N})\bigr)\) 
whose closure \(\overline{G}\) consists entirely of nontrivial Gleason parts, 
so that \(\overline{G}\) is a disjoint union of analytic disks.

For \(f\in H^\infty(\mathbb{D}\times\mathbb{N})\), let 
\(\mathcal{C}_{\hat f}\) denote the set of all analytic disks
in \(\overline{G}\) on which the Gelfand transform \(\hat f\) is constant.

\smallskip

\noindent\textsl{Claim.}
Each \(\mathcal{C}_{\hat f}\) is a nonempty compact subset of \(\overline{G}\).

\smallskip

Indeed, the function \(|\hat f|\), being continuous on the compact set \(\overline{G}\),
attains its maximum at some point \(\eta\).
Let \(G_\eta\) be the nontrivial Gleason part containing \(\eta\),
and let \(\pi:\mathbb{D}\to G_\eta\) be its analytic parametrization
(as in the above definition).
Then the function \(|\hat f \circ \pi|\) is subharmonic on \(\mathbb{D}\)
and attains its maximum at the interior point \(\pi^{-1}(\eta)\).
By the maximum modulus principle, it must be constant,
and hence \(\hat f \circ \pi\), and therefore \(\hat f\),
is constant on \(G_\eta\).
Thus \(G_\eta \subset \mathcal{C}_{\hat f}\),
showing that \(\mathcal{C}_{\hat f}\) is nonempty.

\smallskip

Next, we show that \(\mathcal{C}_{\hat f}\) is compact.
Let \(\theta\) be a limit point of \(\mathcal{C}_{\hat f}\).
By the hypothesis on \(\overline{G}\), the point \(\theta\) belongs to some
maximal analytic disk, i.e., to the Gleason part of \(\theta\).
By~\cite[Thm.\,2.5\,(b),(d)]{Brudnyi2021},
there exists an open neighborhood \(U\) of \(\theta\) in the set of nontrivial
Gleason parts of \(H^\infty(\mathbb{D}\times\mathbb{N})\),
and a homeomorphism
\[
h : S\times \mathbb{D} \longrightarrow U,
\]
where \(S\) is a compact subset of the Stone–\v{C}ech compactification
\(\beta\mathbb{N}\),
such that each slice \(h(\{\omega\}\times \mathbb{D})\) lies in a nontrivial
Gleason part of \(H^\infty(\mathbb{D}\times\mathbb{N})\),
and for every \(f\in H^\infty(\mathbb{D}\times\mathbb{N})\) the pullback
\(\hat f\circ h(\omega,\cdot)\) belongs to \(H^\infty(\mathbb{D})\).

Since \(\theta\) is a limit point of \(\mathcal{C}_{\hat f}\),
there exists a net \((\theta_\alpha)\subset
\mathcal{C}_{\hat f}\cap U\) converging to \(\theta\).
If \(\theta_\alpha\) lies in the image under \(h\)
of the slice corresponding to \(\omega_\alpha\),
then \(\hat f\) is constant on this image.
Passing to the limit within \(U\) shows that \(\hat f\) is
constant on the image under \(h\) of the slice containing \(\theta\).
By analyticity, \(\hat f\) is constant on the entire Gleason part
containing \(\theta\); hence \(\theta\in\mathcal{C}_{\hat f}\).

Thus \(\mathcal{C}_{\hat f}\) is closed in \(\overline{G}\),
and compactness follows since \(\overline{G}\) is compact.

\smallskip

Further, given any finite collection \(f_1,\dots,f_k \in H^\infty(\mathbb{D}\times\mathbb{N})\),
apply the maximum modulus principle to the continuous function
\[
\hat{F} := \sum_{j=1}^k |\hat{f_j}|^2
\]
restricted to \(\overline{G}\).
As in the case of a single function considered in the proof of the \textsl{Claim},
the function \(|\hat{F}|\), being continuous on the compact set \(\overline{G}\),
attains its maximum at some point \(\eta\).
Let \(G_\eta\) be the nontrivial Gleason part containing \(\eta\),
and let \(\pi:\mathbb{D}\to G_\eta\) be its analytic parametrization. 
Then the function \(|\hat{F}\circ \pi|\) is subharmonic on \(\mathbb{D}\)
and attains its maximum at the interior point \(\pi^{-1}(\eta)\).
By the maximum modulus principle, it must be constant,
and therefore \(\hat{F}\circ \pi\),
being a holomorphic map from \(\mathbb{D}\) into a Euclidean sphere in \(\mathbb{C}^k\),
is constant. Hence each \(\hat{f_j}\) is constant on \(G_\eta\),
and consequently
\[
\mathcal{C}_{\hat{f_1}} \cap \cdots \cap \mathcal{C}_{\hat{f_k}} \neq \varnothing .
\]

\smallskip

Since each \(\mathcal{C}_{\hat f}\) is compact and all finite intersections are nonempty,
the finite intersection property implies that
\[
\bigcap_{f\in H^\infty(\mathbb{D}\times\mathbb{N})} \mathcal{C}_{\hat f} \neq \varnothing .
\]
Hence there exists an analytic disk in \(\overline{G}\)
on which \(\hat f\) is constant for every
\(f \in H^\infty(\mathbb{D}\times\mathbb{N})\).
This yields a contradiction, since the Gelfand transforms of functions from the algebra 
\(H^\infty(\mathbb{D}\times\mathbb{N})\) separate points of its maximal ideal space.
Therefore \(\overline{G}\) cannot consist solely of analytic disks, 
that is, it must contain a trivial Gleason part, as claimed.
\end{proof}

\smallskip

Since \(F_\xi\) is contained in the homeomorphic image of 
\(M\bigl(H^\infty(\mathbb{D}\times\mathbb{N})\bigr)\) 
under \(R^*\), which preserves the analytic structure,
the subset of trivial Gleason parts in \(F_\xi\) is nonempty.
This completes the proof of part~\textup{(ii)} of Theorem~\ref{teo2.3}.
\end{proof}

\begin{remark}\label{rem5.4}
It is not difficult to construct points \(\xi\) as in part~\textup{(ii)} 
for which the corresponding fibers \(F_\xi\) contain analytic disks.
Whether this holds for all such \(\xi\) remains an open question.
We conjecture that every fiber of this type contains analytic disks.
\end{remark}

\section{Proof of Theorem \ref{te2.5}}
We begin with the following auxiliary result.

\begin{lemma}\label{lem6.1}
Let \(H=(H_1,\dots,H_m)\in[H^\infty(\mathbb D^n)]^m\). The following are equivalent:

\smallskip
\noindent\textup{(1)}
For each \(x\in\mathbb D^n\) there exist a neighborhood \(U_x\subset\mathbb D^n\)
of \(x\) and a one-dimensional complex analytic subset
\(\Sigma_x\subset\mathbb C^m\) such that \(H(U_x)\subset\Sigma_x\).

\smallskip
\noindent\textup{(2)}
There exist holomorphic maps
\[
h:\mathbb D^n\to\mathbb D,\qquad 
\widetilde H:\mathbb D\to\mathbb C^m
\]
such that \(H=\widetilde H\circ h\).
\end{lemma}

\begin{proof}
We assume that \(H\) is nonconstant; otherwise both statements are trivial.

\smallskip\noindent
\emph{(1) \(\Rightarrow\) (2).}
Choose radii \(0<r_1<r_2<\cdots\nearrow1\) and consider the exhaustion of
\(\mathbb D^n\) by open polydisks \(\mathbb D_{r_j}^n\) of radii \(r_j\).
By compactness of the closure \(\overline{\mathbb D}_{r_j}^{\,n}\), there exist
\(x_1,\dots,x_N\) with
\[
\overline{\mathbb D}_{r_j}^{\,n}
\subset U_{x_1}\cup\cdots\cup U_{x_N},
\]
hence
\[
H(\mathbb D_{r_j}^n)\subset
\Sigma_j:=\Sigma_{x_1}\cup\cdots\cup\Sigma_{x_N},
\]
a one-dimensional analytic set.

Let \(\nu_j:\Sigma_j^\nu\to\Sigma_j\) be the normalization
(\cite[Ch.~8]{GrauertRemmert1984}).  
By its universal property, the holomorphic map
\(H|_{\mathbb D_{r_j}^n}\) lifts to a holomorphic map
\[
H_j^\nu:\mathbb D_{r_j}^n \longrightarrow \Sigma_j^\nu,
\qquad
H = \nu_j \circ H_j^\nu.
\]
Since \(H_j^\nu\) is nonconstant and \(\mathbb D_{r_j}^n\) is connected,
the open mapping theorem implies that
\[
R_j := H_j^\nu(\mathbb D_{r_j}^n)
\]
is a connected open Riemann surface.

Set
\[
h_j:=H_j^\nu:\mathbb D_{r_j}^n\to R_j,\qquad
\widetilde H_j:=\nu_j|_{R_j}:R_j\to\mathbb C^m,
\]
so \(H|_{\mathbb D_{r_j}^n}=\widetilde H_j\circ h_j\).
If \(r_i<r_j\), uniqueness of normalization factorizations yields a unique
holomorphic map \(\varphi_{ij}:R_i\to R_j\) such that
\(h_j=\varphi_{ij}\circ h_i\) on \(\mathbb D_{r_i}^n\).
Thus \(\{R_j,\varphi_{ij}\}\) forms a direct system of Riemann surfaces.

\smallskip
\textit{Formation of the limit surface.}
Define the inductive limit
\[
R:=\varinjlim R_j.
\]
Each \(\varphi_{ij}\) is a local biholomorphism on the relevant component of
\(R_i\), so the complex structures are compatible under pullback.  
Hence \(R\) becomes a Hausdorff one-dimensional complex manifold.
Since each \(R_j\) is connected and the overlaps
\(\varphi_{ij}(R_i)\cap R_j\) are nonempty open subsets,
the union \(R = \bigcup_j R_j\) is connected, and hence \(R\) is a  connected Riemann surface.

Since the factorizations are compatible,
the \(h_j\) glue to a holomorphic map
\[
h_0:\mathbb D^n\to R,
\]
and the \(\widetilde H_j\) glue to a holomorphic map
\[
\widetilde H_0:R\to\mathbb C^m,
\]
yielding
\[
H=\widetilde H_0\circ h_0.
\]

\smallskip
\textit{Uniformization.}
As \(\widetilde H_0\) is bounded and nonconstant,  \(R\)
admits a nonconstant bounded holomorphic function, hence \(R\) is hyperbolic.
Thus its universal cover is \(\mathbb D\).  
Let \(p:\mathbb D\to R\) be the covering map.  
Since \(\mathbb D^n\) is simply connected,
the covering homotopy theorem implies the existence of a holomorphic map
\[
h:\mathbb D^n\to\mathbb D,\qquad h_0=p\circ h,
\]
and setting \(\widetilde H:=\widetilde H_0\circ p\) gives
\(H=\widetilde H\circ h\).

\smallskip\noindent
\emph{(2) \(\Rightarrow\) (1).}
Assume \(H = \widetilde H \circ h\) with
\(h:\mathbb D^n\to\mathbb D\) and 
\(\widetilde H:\mathbb D\to\mathbb C^m\) nonconstant holomorphic maps.
Fix \(x \in \mathbb D^n\), set \(\zeta_0 := h(x)\), \(y_0 := H(x)\).

Since \(h\) is nonconstant, the open mapping theorem implies that
\(h(\mathbb D^n)\) is an open subset of \(\mathbb D\).
Hence we can choose an open disk \(W_0 \subset h(\mathbb D^n)\) centered at \(\zeta_0\)
and set
\[
W := W_0 - \zeta_0 = \{\zeta - \zeta_0 : \zeta \in W_0\},
\]
so that \(W\) is an open disk centered at \(0\).

Define
\[
\varphi(\zeta) := \widetilde H(\zeta + \zeta_0) - y_0,
\qquad \zeta \in W,
\]
so that \(\varphi(0) = 0\). 
Write \(\varphi = (\varphi_1,\dots,\varphi_m)\) and set
\[
k := \min_{1\le j\le m} \operatorname{ord}_{0}\varphi_j < \infty.
\]
By permuting the coordinates in \(\mathbb C^m\), if necessary, and precomposing
\(\varphi\) with a biholomorphic change of variable on a smaller disk \(W'\subset W\)
fixing \(0\) (and then renaming \(W'\) by \(W\)), we may assume without loss
of generality that
\[
\varphi_1(\zeta) = \zeta^k,\qquad
\varphi_j(\zeta) = \zeta^k \psi_j(\zeta),\quad j\ge2,\ \ \zeta\in W,
\]
with each \(\psi_j\) holomorphic near \(0\).
Thus
\[
\varphi(\zeta)
= \bigl(\zeta^k,\,\zeta^k\psi_2(\zeta),\dots,\zeta^k\psi_m(\zeta)\bigr),
\qquad \zeta\in W.
\]

Consider the projection onto the first coordinate,
\[
\pi_1 : \mathbb C^m \to \mathbb C,\qquad \pi_1(z_1,\dots,z_m)=z_1.
\]
Then for all \(\zeta\in W\),
\[
\pi_1\bigl(\varphi(\zeta)\bigr) = \zeta^k.
\]
Hence for some \(r>0\),
\[
\pi_1\circ\varphi : W\setminus\{0\} \longrightarrow \{w\in\mathbb C : 0<|w|<r\}
\]
is a finite unbranched covering of degree \(k\). Consequently,
\[
\widetilde H(W_0)=\varphi(W) + y_0
\]
is a one-dimensional complex analytic subset of a neighborhood
of \(y_0\in\mathbb C^m\).

Finally, let \(U_x := h^{-1}(W_0)\). Then \(U_x\) is a neighborhood of \(x\)
and
\[
H(U_x) = \widetilde H(h(U_x)) = \widetilde H(W_0),
\]
establishing \textup{(1)}.
\end{proof}

\medskip

\noindent\textit{Proof of Theorem~\ref{te2.5}.}
We must show that whenever \(g_1,\dots,g_k\in A(S_n;H)\), with \(H\) as in the
hypotheses of the theorem, satisfy the corona condition
\[
\inf_{z\in\mathbb D^n}\ \max_{1\le i\le k}|g_i(z)| \ge \delta>0,
\]
then for every \(h_1,\dots,h_k\in H^\infty(\mathbb D^n)\) with
\(\max_i\lVert h_i\rVert_\infty<\delta\), the ideal generated by
\[
F_i := g_i + h_i,\qquad i=1,\dots,k,
\]
is the whole algebra \(H^\infty(\mathbb D^n)\).

Suppose, to the contrary, that the ideal generated by
\(F_1,\dots,F_k\) is proper.  
Then there exists a maximal ideal 
\(\chi\in M\bigl(H^\infty(\mathbb D^n)\bigr)\) containing it; hence
\[
\widehat F_i(\chi)=0,\qquad i=1,\dots,k.
\]

Let
\[
\pi_n:M\bigl(H^\infty(\mathbb D^n)\bigr)\longrightarrow 
       \bigl[M\bigl(H^\infty(\mathbb D)\bigr)\bigr]^n
\]
be the canonical projection introduced in \eqref{e2.2}, and set 
\(\xi:=\pi_n(\chi)\).
The associated fiber is
\[
F_\xi := \pi_n^{-1}(\xi)\cap\operatorname{cl}(\mathbb D^n).
\]
Let \(A(F_\xi)\) denote the uniform closure in \(C(F_\xi)\) of the
restrictions of \(H^\infty(\mathbb D^n)\) to \(F_\xi\).
By the first assertion of Theorem~\ref{teo2.2} we have
\[
\pi_n^{-1}(\xi)=M\bigl(A(F_\xi)\bigr),
\]
so, in particular, \(\chi\in M\bigl(A(F_\xi)\bigr)\).
We therefore regard \(\chi\) as a character of \(A(F_\xi)\), and continue to
denote it by \(\chi\).

By Lemma~\ref{lem6.1}, there exist holomorphic maps
\(h:\mathbb D^n\to\mathbb D\) and \(\widetilde H:\mathbb D\to\mathbb C^m\)
such that \(H=\widetilde H\circ h\). Since \(H\) is bounded on \(\mathbb D^n\),
it follows that \(\widetilde H\) is bounded on \(\mathbb D\).

The map \(h\) induces a homomorphism
\[
h^*:H^\infty(\mathbb D)\longrightarrow H^\infty(\mathbb D^n),\qquad 
\varphi\longmapsto \varphi\circ h,
\]
whose transpose determines a continuous map of maximal ideal spaces
\[
\hat h^*:M\bigl(H^\infty(\mathbb D^n)\bigr)\longrightarrow
          M\bigl(H^\infty(\mathbb D)\bigr).
\]

Set
\[
K := \hat h^*(F_\xi)\subset M\bigl(H^\infty(\mathbb D)\bigr).
\]

Since the Gelfand transform of every function in \(S_n(H^\infty(\mathbb D))\)
is constant on the fibres of \(\pi_n\), the restriction of each \(\hat g_i\)
to \(F_\xi\) is the uniform limit in \(C(F_\xi)\) of the restrictions to \(F_\xi\)
of Gelfand transforms of a sequence of holomorphic polynomials in the coordinate
functions of \(H\), that is, polynomials in \(H_1,\dots,H_m\). Using the factorization \(H=\widetilde H\circ h\), each such polynomial in
\(H_1,\dots,H_m\) can be written as a polynomial in the coordinate functions of
\(\widetilde H\) composed with \(h\).  Hence its Gelfand transform descends,
via \(\hat h^*\), to the Gelfand transform of a holomorphic polynomial on
\(\mathbb D\) in the coordinates of \(\widetilde H\).

Therefore the restriction of each \(\hat g_i\) to \(F_\xi\) factors through
\(\hat h^*\); that is, there exist continuous functions
\(\widetilde g_i\in C(K)\) such that
\[
\hat g_i = \widetilde g_i \circ \hat h^*
\quad\text{on }\  F_\xi,\quad i=1,\dots,k.
\]
Moreover, each \(\widetilde g_i\) is the uniform limit on \(K\) of Gelfand 
transforms of functions from \(H^\infty(\mathbb D)\).

Let \(\mathcal A(K)\) denote the uniform closure on \(K\) of the algebra of 
restrictions of continuous functions defined on open neighbourhoods \(U\) of \(K\) 
whose restrictions to \(U\cap\mathbb D\) are holomorphic.  
By results of Suárez 
(\cite[Th.~2.4 and Lem.~1.1]{Suarez1998} together with \cite[Th.~2.4]{Suarez1994}), 
the maximal ideal space of \(\mathcal A(K)\) is canonically identified with \(K\).

The restriction of \(\hat h^*\) to \(F_\xi\) induces a pullback homomorphism
\[
\Phi:\mathcal A(K)\longrightarrow A(F_\xi),\qquad 
f\longmapsto f\circ\hat h^*\big|_{F_\xi}.
\]
Indeed, if \(U\) is an open neighbourhood of \(K\) and \(f\) is continuous on \(U\)
with \(f|_{U\cap\mathbb D}\) holomorphic, then 
\((\hat h^*)^{-1}(U)\) is an open neighbourhood of \(F_\xi\) in 
\(M\bigl(H^\infty(\mathbb D^n)\bigr)\), and
\(f\circ\hat h^*\) is continuous on this neighbourhood and holomorphic on its
intersection with \(\mathbb D^n\) (since \(\hat h^*(z)=h(z)\) for \(z\in\mathbb D^n\)).  
Hence the restriction of \(f\circ\hat h^*\) to \(F_\xi\) lies in \(A(F_\xi)\) 
(cf. the definition in Section~\ref{main} after \eqref{e2.3}).  
Since \(\mathcal A(K)\) is the uniform closure of such functions \(f\) on \(K\), 
it follows that \(\Phi(\mathcal A(K))\subset A(F_\xi)\).

The transpose of \(\Phi\),
\[
\Phi^*:M(A(F_\xi))\longrightarrow M(\mathcal A(K))\cong K,
\]
is therefore well defined and continuous.  Moreover, \(\Phi^*\) coincides with the 
restriction of \(\hat h^*\) to \(M(A(F_\xi))=\pi_n^{-1}(\xi)\).  
Indeed, for any \(f\in H^\infty(\mathbb D)\),
\[
\Phi(\hat f|_K)
   = \hat f\circ\hat h^*\big|_{\pi_n^{-1}(\xi)}
   = \widehat{f\circ h}\big|_{\pi_n^{-1}(\xi)}
   = \widehat{h^*(f)}\big|_{\pi_n^{-1}(\xi)},
\]
and evaluation at \(\chi\in M(A(F_\xi))\) gives
\[
\Phi^*(\chi)(\hat f|_K)
   = \chi\bigl(\widehat{h^*(f)}\bigr)
   = \hat f\bigl(\hat h^*(\chi)\bigr),
\]
which means precisely that \(\Phi^*(\chi)=\hat h^*(\chi)\).
Thus we may identify \(\Phi^*\) with the restriction of \(\hat h^*\), and we
continue to denote it by
\[
\hat h^*:M(A(F_\xi))\longrightarrow K.
\]

In particular, for our chosen \(\chi\in M(A(F_\xi))\),
\[
\eta:=\hat h^*(\chi)\in K,
\qquad
\hat g_i(\chi)=\widetilde g_i(\eta),\quad i=1,\dots,k.
\]

The corona condition now gives
\[
\max_{1\le i\le k}|\widetilde g_i(\eta)|
 =\max_{1\le i\le k}|\hat g_i(\chi)|
 \,\ge\,\delta.
\]
Hence for some index \(i_0\),
\[
|\hat g_{i_0}(\chi)|\ge\delta.
\]
On the other hand,
\[
|\hat h_{i_0}(\chi)|\le\|h_{i_0}\|_\infty<\delta.
\]
Therefore
\[
|\hat F_{i_0}(\chi)|
 =|\hat g_{i_0}(\chi)+\hat h_{i_0}(\chi)|
 \ge |\hat g_{i_0}(\chi)|-|\hat h_{i_0}(\chi)|
 >\delta-\delta=0,
\]
which contradicts \(\hat F_{i_0}(\chi)=0\).

Thus no maximal ideal of \(H^\infty(\mathbb D^n)\) can contain the ideal generated
by \(F_1,\dots,F_k\), and hence this ideal is all of
\(H^\infty(\mathbb D^n)\).  
The B\'ezout equation for \(F_1,\dots,F_k\) therefore has a solution in
\(H^\infty(\mathbb D^n)\), completing the proof.
\qedhere

\end{document}